\documentclass[11pt]{article}
\usepackage{amsmath, amscd, amssymb, latexsym, epsfig, color, amsthm}
\numberwithin{equation}{section}

\newtheorem{theorem}{Theorem}[section]

\newtheorem{corollary}[theorem]{Corollary}
\newtheorem{conjecture}[theorem]{Conjecture}
\newtheorem{lemma}[theorem]{Lemma}

\newtheorem{problem}[theorem]{Problem}

\theoremstyle{definition}

\DeclareMathOperator{\Skel}{Skel}
\DeclareMathOperator\lk{\mathrm{lk}}

\DeclareMathOperator{\conv}{\mathrm{conv}}

\DeclareMathOperator{\aff}{\mathrm{aff}}
\DeclareMathOperator{\Hilb}{\mathrm{Hilb}}
\DeclareMathOperator{\Rig}{Rig}
\DeclareMathOperator{\rank}{rank}

\newcommand{\field}{{\mathbb K}}
\newcommand{\N}{{\mathbb N}}
\newcommand{\R}{{\mathbb R}}

\newcommand{\Z}{{\mathbb Z}}

\newcommand{\C}{{\mathcal C}}
\newcommand{\T}{{\mathcal T}}
\newcommand{\B}{{\mathcal B}}
\newcommand{\M}{{\mathcal M}}
\newcommand{\Stress}{{\mathcal S}}
\newcommand{\p}{\rho}
\newcommand{\m}{\mu}

\newcommand{\fmax}{\ensuremath{\mathrm{fmax}}\hspace{1pt}}

\newcommand{\Int}{\mbox{\upshape int}\,}
\newcommand{\vol}{\mbox{\upshape Vol}\,}

\title{A tale of centrally symmetric polytopes and spheres}
\author{Isabella Novik\thanks{Research is partially\textsl{}
supported by NSF grants DMS-1361423 and DMS-1664865, and by Robert R.~\&  Elaine F.~Phelps Professorship in Mathematics. This material is based on work supported by the National Science Foundation under Grant No.~DMS-1440140 while the author was in residence at the Mathematical Sciences Research Institute in Berkeley CA, during the Fall 2017
semester.
}\\
\small Department of Mathematics\\[-0.8ex]
\small University of Washington\\[-0.8ex]
\small Seattle, WA 98195-4350, USA\\[-0.8ex]
\small \texttt{novik@math.washington.edu}
}

\begin{document}
\maketitle

\begin{abstract} This paper is a survey of recent advances as well as open problems in the study of face numbers of centrally symmetric simplicial polytopes and spheres. The topics discussed range from neighborliness of centrally symmetric polytopes and the upper bound theorem for centrally symmetric simplicial spheres to the generalized lower bound theorem for centrally symmetric simplicial polytopes and the lower bound conjecture for centrally symmetric simplicial spheres and manifolds.
\end{abstract}

\section{Introduction}
The goal of this paper is to survey recent results related to face numbers  of centrally symmetric simplicial polytopes and spheres. To put things into perspective, we start by discussing simplicial polytopes and spheres without a symmetry assumption. The classical theorems of Barnette \cite{Barnette73} and McMullen \cite{McMullen70}, known as the Lower Bound Theorem (LBT, for short) and the Upper Bound Theorem  (or UBT), assert that among all $d$-dimensional simplicial polytopes with $n$ vertices a stacked polytope simultaneously minimizes all the face numbers while the cyclic polytope simultaneously maximizes all the face numbers. Furthermore, the same results continue to hold in the generality of $(d-1)$-dimensional simplicial spheres (see \cite{Barnette-LBT-pseudomanifolds} and \cite{Stanley75}). It is also worth mentioning that in the class of simplicial spheres of dimension $d-1\geq 3$ with $n$ vertices, the (boundary complexes of the) stacked polytopes are the only minimizers \cite{Kalai87}. On the other hand, the maximizers are precisely the $\lfloor d/2\rfloor$-neighborly spheres --- a set that, in addition to the (boundary complex of the) cyclic polytope, includes many other simplicial spheres, see, for instance, a recent paper \cite{Padrol-13}.  

In fact, much more is known: the celebrated $g$-theorem of Billera--Lee \cite{BilleraLee} and Stanley \cite{Stanley80} provides a complete characterization of all possible $f$-vectors of simplicial polytopes. The equally celebrated $g$-conjecture posits that the same characterization is valid for the $f$-vectors of simplicial spheres.  (The $f$-vector encodes the number of faces of each dimension  of a simplicial complex.) In other words, at least conjecturally, the $f$-vector cannot differentiate between simplicial polytopes and simplicial spheres. 

Let us now restrict our world and consider all {\em centrally symmetric} (cs, for short) simplicial polytopes and all {\em centrally symmetric} simplicial spheres. What is known about the $f$-vectors of these objects? Are there cs analogs of the LBT and UBT, and ultimately of the $g$-theorem and $g$-conjecture, respectively? Surprisingly very little is known, and not for the lack of effort. There is an analog of the LBT for cs simplicial polytopes established by Stanley \cite{Stanley-87} as well as a characterization of minimizers \cite{KNNZ}, and while it is plausible that the same results continue to hold for all cs simplicial spheres, the proofs remain elusive. 

Still, the most mysterious and fascinating side of the story comes from trying to understand the upper bounds: in contrast with the situation for simplicial polytopes and spheres without a symmetry assumption, the existing upper-bound type results and conjectures indicate striking differences between $f$-vectors of cs polytopes and those of cs spheres. In other words, cs spheres and polytopes do not look alike $f$-wise, even though their non-cs counterparts do! For instance, the (appropriately defined) neighborliness of cs polytopes is quite restrictive \cite{LinNov}, and a cs $d$-dimensional polytope with more than $2^d$ vertices cannot be even  $2$-neighborly; yet, by a result of Jockusch  \cite{Jockusch95}, there exist cs $2$-neighborly simplicial spheres of dimension $3$ with any even number $n\geq 8$ of vertices. Thus, for a sufficiently large $n$, the maximum number of edges that a  cs $4$-dimensional polytope with $n$ vertices  can have \emph{differs} from the maximum number of edges that a cs simplicial sphere of dimension $3$ with $n$ vertices can have; moreover, the former quantity (or even its asymptotic behavior) is unknown at present. This indicates how very far we are from even posing a plausible (sharp) Upper Bound Conjecture for cs $d$-dimensional polytopes with $d\geq 4$. At the same time,  Adin \cite{Adin91} and Stanley (unpublished) provided upper bounds on face numbers of cs simplicial spheres of dimension $d-1$ with $n$ vertices; these bounds are attained by a cs $\lfloor d/2\rfloor$-neighborly sphere of dimension $d-1$ with $n$ vertices \emph{assuming} such a sphere exist.

Aside from being of intrinsic interest, additional motivation to better understand the $f$-vectors of cs polytopes arises from the recently discovered tantalizing connections (initiated by Donoho and his collaborators, see, for instance, \cite{Don06a, DonTan10}) between cs polytopes  with many faces and seemingly unrelated areas of error-correcting codes and sparse signal reconstruction. Furthermore, any cs convex body in $\R^d$ is the unit ball of a certain norm on $\R^d$. As a result, methods used in the study of face numbers of cs complexes, at present, involve not only commutative algebra (via the study of associated Stanley--Reisner rings) but also a wealth of techniques from  geometric analysis. Yet, many problems remain notoriously difficult. 

The goal of this paper is thus to survey the few known results on $f$-vectors of cs simplicial polytopes and spheres, showcase several existing techniques, and, most importantly, present many open problems. It is our hope that collecting such problems in one place will catalyze progress  in this fascinating field. For a quick preview of what is known and what is not, see the following table.

\begin{center}
\tiny
\begin{tabular}{|c|c|c|c||c|c|c|} \hline 
& \multicolumn{3}{c||}{\bf non-cs} & \multicolumn{3}{c|}{\bf cs}\\ \hline
 & \begin{tabular}{c}simplicial \\ polytopes \end{tabular} & 
	 \begin{tabular}{c}simplicial \\ spheres \end{tabular} & same? & 
	 \begin{tabular}{c}simplicial \\ polytope\end{tabular}&
	 \begin{tabular}{c}simplicial \\ spheres\end{tabular} & same? \\ \hline\hline
UBT & $\surd$ &  $\surd$ & yes & 
  \begin{tabular}{c}no plausible\\ conjecture\end{tabular} & 
	\begin{tabular}{c}known bounds\\ conjecturally sharp\end{tabular} & no \\ \hline
LBT & $\surd$ &  $\surd$ & yes & $\surd$ & conjecture & 
   \begin{tabular}{c}conjecturally\\ yes\end{tabular}\\ \hline
GLBT & $\surd$ &  conjecture & 
  \begin{tabular}{c}conjecturally\\ yes\end{tabular} & $\surd$ & conjecture & \begin{tabular}{c}conjecturally\\ yes\end{tabular}\\ \hline
\end{tabular}
\end{center}

The rest of the paper is structured as follows. In Section 2, we set up notation and recall basic definitions pertaining to simplicial polytopes and spheres. Section 3 is devoted to neighborliness of cs polytopes.  This leads to discussion of upper-bound type results and problems on face numbers of cs polytopes, see Section 4.  Section 5 deals with neighborliness and upper-bound type results for cs simplicial spheres. Section 6 takes us into the algebraic side of the story: there we present a quick review of the Stanley--Reisner ring --- the major algebraic tool in the study of face numbers, sketch the proof of the classical UBT along with Adin--Stanley's variant of this result for cs spheres, and prepare the ground for the following sections. Sections 7 and 8 are concerned with the lower-bound type results and conjectures. Specifically, in Section 7 we discuss a cs analog of the Generalized LBT for cs simplicial polytopes --- a part of the story that is most well understood, while in Section 8 we consider a natural conjectural cs analog of the LBT for spheres, manifolds, and pseudomanifolds. We close with a few concluding remarks in Section 9.

\section{Preliminaries} 
We start with outlining basic definitions and notation we will use throughout the paper. A {\em polytope} is the convex hull of a set of finitely many points in $\R^d$. One example is the (geometric) {\em $d$-simplex} defined as the convex hull of an arbitrary set of $d+1$ affinely independent points in $\R^d$.  A (proper) {\em face} of any convex body $K$ (e.g., a polytope) is the intersection of $K$ with a supporting affine hyperplane; see, for example, Chapter II of \cite{Ba02}. 
A polytope $P$ is called {\em simplicial} if all of its (proper) faces are simplices. The {\em dimension} of a polytope $P$ is the dimension of the affine hull of $P$. We refer to $d$-dimensional polytopes as $d$-polytopes, and to $i$-dimensional faces as $i$-faces. 

A polytope $P \subset \R^d$ is {\em centrally symmetric} ({\em cs}, for short) if $P = -P$; that is, $x\in P$ if and only if $-x\in P$. An important example of a cs polytope is the $d$-dimensional {\em cross-polytope} $\C^*_d=\conv(\pm p_1,\pm p_2,\ldots, \pm p_d)$, where $p_1,p_2,\ldots, p_d$ are points in $\R^d$ whose position vectors form a basis for $\R^d$. If these position vectors form the standard basis of $\R^d$, denoted $\{e_1,\ldots, e_d\}$, then the resulting polytope is the unit ball of the $\ell^1$-norm; we  refer to this particular instance of the cross-polytope as the {\em regular} cross-polytope.

A {\em simplicial complex}  $\Delta$ on a (finite) vertex set $V=V(\Delta)$ is a collection of subsets of $V$ that is closed under inclusion; an example is the (abstract) simplex on $V$, $\overline{V}:=\{F : F\subseteq V\}$. The elements of $\Delta$ are called {\em faces} and the maximal under inclusion faces are called {\em facets}. The {\em dimension} of a face $F\in\Delta$ is $\dim F=|F|-1$, and the dimension of $\Delta$ is $\dim\Delta:=\max\{\dim F \ : \ F\in\Delta\}$. The {\em $k$-skeleton} of $\Delta$, $\Skel_k(\Delta)$, is a subcomplex of $\Delta$ consisting of all faces of dimension $\leq k$. The {\em $f$-vector} of a simplicial complex $\Delta$ is $f(\Delta):=(f_{-1}(\Delta), f_{0}(\Delta), \ldots, f_{\dim\Delta}(\Delta)),$ where $f_i=f_i(\Delta)$ denotes the number of $i$-faces of $\Delta$; the numbers $f_i$ are called the {\em $f$-numbers} of  $\Delta$.  

Each simplicial complex $\Delta$ admits a geometric realization $\|\Delta\|$ that contains a geometric $i$-simplex for each $i$-face of $\Delta$.  
A simplicial complex $\Delta$ is a {\em simplicial sphere} ({\em simplicial ball}, or {\em simplicial manifold}, respectively) if $\|\Delta\|$ is homeomorphic to a sphere (ball, or closed manifold, respectively). If $P$ is a simplicial $d$-polytope, then  the empty set along with the collection of the vertex sets of all the (proper) faces of $P$ is a  simplicial sphere of dimension $d-1$ called the {\em boundary complex} of $P$; it is denoted by $\partial P$. While it follows from Steinitz' theorem that every simplicial $2$-sphere can be realized as the boundary complex of a simplicial $3$-polytope, for $d\geq 4$, there are many more combinatorial types of simplicial $(d-1)$-spheres than those of boundary complexes of simplicial $d$-polytopes, see \cite{Kal, NeSanWil, PfeiZieg}. 

A simplicial complex $\Delta$ is {\em centrally symmetric} (or cs)  if it is equipped with a simplicial {\em involution} $\phi: \Delta \to \Delta$ such that for every non-empty face $F\in\Delta$, $\phi(F)\neq F$. We refer to $F$ and  $\phi(F)$ as {\em antipodal} faces. For instance, the boundary complex of any cs simplicial polytope $P$ is a cs simplicial sphere with the involution $\phi$ induced by the map $\phi(v)=-v$ on the vertices of $P$. 

A simplicial complex $\Delta$ is {\em $k$-neighborly} if  every set of $k$ of its vertices forms a face of $\Delta$. Equivalently, a simplicial complex $\Delta$ with $n$ vertices is $k$-neighborly if its $(k-1)$-skeleton coincides with the $(k-1)$-skeleton of the $(n-1)$-simplex. Since in a cs complex, a vertex and its antipode  can never form an edge,  this definition of neighborliness requires a suitable adjustment for cs complexes. We say that a cs simplicial complex $\Delta$ is {\em  $k$-neighborly} if  every set of $k$ of its vertices, no two of which are antipodes, forms a face of $\Delta$. Equivalently, a cs simplicial complex $\Delta$ on $2m$ vertices is $k$-neighborly if its $(k-1)$-skeleton coincides with the $(k-1)$-skeleton of $\partial\C^*_m$. 
In particular, if $\Delta$ is a $k$-neighborly simplicial complex, then $f_{j-1}(\Delta)={f_0 \choose j}$ for all $j\leq k$, while if $\Delta$ is a cs $k$-neighborly simplicial complex, then $f_{j-1}(\Delta)=2^j{f_0/2 \choose j}$ for all $j\leq k$. 

It is worth mentioning that similar definitions apply to \emph{general} (i.e., not necessarily simplicial) polytopes. Specifically, a polytope $P$ is $k$-neighborly if every set of $k$ of its vertices forms the vertex set of a face of $P$; a cs polytope $P$ is $k$-neighborly if  every set of $k$ of its vertices, no two of which are antipodes, forms the vertex set of a face of $P$. In the next two sections we work with general cs polytopes.


\section{How neighborly can a cs polytope be?} 

Our story begins with the {\em  cyclic polytope}, $C_d(n)$, which is defined as the convex hull of $n\geq d+1$ distinct points on the {\em moment curve} $\{(t, t^2,\ldots,t^d)\in\R^d \, : \, t\in \R\}$ or on the {\em trigonometric moment curve} $\{(\cos t, \sin t, \cos 2t, \sin 2t,\ldots, \cos kt, \sin kt)\in\R^{2k} \, : \, t\in \R\}$, assuming $d=2k$. Both types of cyclic polytopes were investigated by Carath\'eodory \cite{Carath1911} and later by Gale \cite{Gale63}, who, in particular, showed that the two types are combinatorially equivalent (assuming $d$ is even) and independent of the choice of points.  In fact, the two types are projectively equivalent, see \cite[Exercise 2.21]{Ziegler}. These polytopes were also rediscovered by Motzkin \cite{Motz57, GruMotz63} and many others; see books \cite{Ba02, Ziegler} for more information on these amazing polytopes. The properties  of the cyclic polytope $C_d(n)$ that will be important for us here  are: it is a (non cs) simplicial $d$-polytope with $n$ vertices; furthermore, it is $\lfloor d/2 \rfloor$-neighborly for all $n\geq d+1$. 

The existence of cyclic polytopes motivated several questions, among them:
do there exist cs $d$-polytopes (apart from the cross-polytope) that are $\lfloor d/2 \rfloor$-neighborly or at least ``highly"  neighborly? This section discusses the state-of-affairs triggered by this question.

It became apparent from the very beginning that the answer is likely to be both interesting and complicated: several works from the sixties indicated that in contrast to the general case, the neighborliness of cs polytopes might be rather restricted. Specifically,  Gr\"unbaum showed in 1967 \cite[p.~116]{Gru-book} that while there exists a cs $2$-neighborly $4$-polytope with $10$ vertices, no cs $4$-polytope with  $12$ or more vertices can be $2$-neighborly. This observation was extended by McMullen and Shephard \cite{McMShep} who proved that while there exists a cs $\lfloor d/2\rfloor$-neighborly $d$-polytope with $2(d+1)$ vertices, a cs $d$-polytope with at least $2(d+2)$ vertices cannot be more than  $\lfloor (d+1)/3\rfloor$-neighborly.  A cs $\lfloor d/2\rfloor$-neighborly $d$-polytope with $2(d+1)$ vertices is easy to construct: for instance, $\conv\big(\pm e_1,\pm e_2,\ldots, \pm e_d, \pm \sum_{i=1}^d e_i\big)$ does the job. On the other hand, to show that a cs $d$-polytope with $2(d+2)$ vertices can only be  $\lfloor (d+1)/3\rfloor$-neighborly and to construct cs polytopes achieving this bound, McMullen and Shephard \cite{McMShep} introduced and studied cs transformations of cs polytopes --- a cs analog of celebrated Gale transforms. (See \cite[Section 5.4]{Gru-book} and \cite[Chapter 6]{Ziegler} for a detailed description of Gale transforms and their applications.)

Let $k(d,n)$ denote  the largest integer $k$ such that there exists a cs $d$-polytope with $2(n + d)$ vertices that is $k$-neighborly.  Influenced by their result, McMullen and Shephard \cite{McMShep} conjectured that, in fact, $k(d,n)\leq \lfloor  (d+n-1)/(n+1)\rfloor$ for all $n\geq 3$. Their conjecture was subsequently refuted by Halsey \cite{Halsey} and then by Schneider \cite{Schneider}, but only for $d \gg n$. Namely, Schneider showed that  
$$
\liminf_{d\rightarrow\infty} \frac{k(d,3)}{d+3} \geq 1-2^{-1/2} \quad \mbox{and} \quad
\liminf_{d\rightarrow\infty} \frac{k(d,n)}{d+n} \geq 0.2390 \mbox{ for all $n$}.
$$
 On the other hand, Burton \cite{Burton} proved that  a cs $d$-polytope with a sufficiently large number of vertices  ($\approx (d/2)^{d/2}$) indeed cannot be  even $2$-neighborly. Burton's proof is surprisingly short and simple: it relies on John's ellipsoid theorem, see for instance \cite[Chapter 3]{Ball}, along with a quantitative version of the observation that any sufficiently large finite subset of the Euclidean unit sphere contains two points that are very close to each other.

To the best of the author's knowledge, McMullen--Shephard's values $k(d,1)=\lfloor \frac{d}{2}\rfloor$ and $k(d,2)=\lfloor \frac{d+1}{3}\rfloor$ (see \cite{McMShep}) remain the \emph{only} known exact values of $k(d,n)$. In particular, $k(4,1)=2$, $k(4,2)=1$, $k(5,1)=k(5,2)=2$, but it appears unknown whether $k(5,3)$ equals $2$ or $1$. On the other hand, the asymptotics of $k(d,n)$ is now well understood:

 \begin{theorem} \label{cs-neighb-thm} There exist absolute constants $C_1, C_2>0$ independent of $d$ and $n$ such that
$$\frac{C_1 d}{1+\log \frac{n+d}{d}} \leq
k(d,n) \leq 1+\frac{C_2 d}{1+\log \frac{n+d}{d}}.
$$
\end{theorem}

\noindent Theorem \ref{cs-neighb-thm} is due to Linial and Novik \cite{LinNov}. A dual version of the lower bound part of this theorem was also independently established by Rudelson and Vershynin \cite{RudVersh05}. 

Two extreme, and hence particularly interesting, cases of Theorem \ref{cs-neighb-thm} deserve a special mention: the case of $k(d,n)$ proportional to $d$ and the case of $k(d,n)=1$. Donoho \cite{Don06a} proved that there exists $\rho>0$ such that for large $d$, the orthogonal projection of the $2d$-dimensional regular cross-polytope onto a $d$-dimensional subspace of $\R^{2d}$, chosen uniformly at random, is at least $\lfloor \rho d\rfloor$-neighborly with high probability and provided numerical evidence that $\rho\geq 0.089$. In other words, for large $d$, $k(d,d)\geq \rho d$. (The estimates from \cite{LinNov} guarantee that $\rho\geq 1/400$.) 
As for the other extreme, Theorem \ref{cs-neighb-thm} shows that the largest number of vertices in a cs $2$-neighborly $d$-polytope is $e^{\Theta(d)}$. In fact, the following more precise result is known, see \cite[Theorem 1.2]{LinNov} and \cite[Theorem 3.2]{BarvLeeN-antipodal}.

\begin{theorem} \label{2-neighb}
A cs $2$-neighborly $d$-polytope has at most $2^{d}$ vertices, i.e., $k(d, 2^{d-1}+1-d)=1$. On the other hand, for any even $d\geq 6$, there exists a cs $2$-neighborly $d$-polytope with $2(\sqrt{3}^d/3 - 1)$ vertices. 
\end{theorem}

Our discussion of $k(d,n)$ is summarized in a table below. 
\begin{center}
\small
\begin{tabular}{|c|c|} \hline 
$n$ & $k(d,n)$ \\ \hline\hline
$1$ & $\big\lfloor\frac{d}{2}\big\rfloor$ \\ \hline
$2$ & $ \big\lfloor \frac{d+1}{3}\big\rfloor$ \\ \hline
$d$ & proportional to $d$ \\ \hline
$ \leq 3^{\lfloor d/2\rfloor-1}-1-d$ & $\geq 2$ \\ \hline
$\geq 2^{d-1}+1-d$ & 1\\ \hline
\end{tabular}
\end{center}

We devote the rest of this section to pointing out some of the main ideas used in the proofs. In particular, the proof of the fact that a cs $2$-neighborly $d$-polytope has at most $2^{d}$ vertices is so short that we cannot avoid the temptation to provide it here. 
 For a polytope $P\subset \R^d$ and a vector $a\in\R^d$, define $P_a:=P+a$ to be the translation of $P$ by $a$,  where ``+" denotes the Minkowski addition. The first step in the proof is the following simple observation from \cite{BarvN}:

\begin{lemma} \label{disj-int}
Let $P\subset \R^d$ be a cs $d$-polytope, and let $u$ and $v$ be vertices of $P$, so that $-v$ is also a vertex of $P$. If the polytopes $P_u$ and $P_v$  have intersecting interiors then the vertices $u$ and $-v$ are not connected by an edge. Consequently, if $P$ is a cs $2$-neighborly polytope with vertex set $V$, then the polytopes $\{P_v \ : \  v\in V\}$ have pairwise disjoint interiors.
\end{lemma}
\begin{proof}
The assumption that $\Int(P_u)\cap\Int(P_v)\neq \emptyset$ implies that there exist $x,y\in\Int(P)$ such that $x+u=y+v$, or equivalently, that $(y-x)/2=(u-v)/2$. Since $P$ is centrally symmetric, and $x,y\in\Int(P)$, the point $q:=(y-x)/2$ is an interior point of $P$. As $q$ is also the barycenter of the line segment connecting $u$ and $-v$, this line segment is not an edge of $P$.
\end{proof}

The rest of the proof that a cs $2$-neighborly $d$-polytope has at most $2^d$ vertices utilizes a volume trick that goes back to Danzer and Gr\"unbaum \cite{DanzGrunb}. If $P\subset\R^d$ is a cs $2$-neighborly $d$-polytope with vertex set $V$, then by Lemma \ref{disj-int}, the polytopes $\{P_v \ : \  v\in V\}$ have pairwise disjoint interiors. Therefore,
\[ \vol\big(\bigcup_{v\in V} P_v\big)=\sum_{v\in V} \vol (P_v)=|V| \cdot \vol(P).\]
On the other hand, since for $v\in V$, $P_v=P+v\subset 2P$, it follows that $\bigcup_{v\in V}P_v \subseteq 2P$, and so
\[ \vol\big(\bigcup_{v\in V} P_v\big) \leq \vol(2P)=2^d \cdot \vol(P).\]
Comparing these two equations yields $|V|\leq 2^d$, as desired.

The proof of the upper bound part of Theorem \ref{cs-neighb-thm} relies on a more intricate application of the same volume trick. Let $P$ be a cs $k$-neighborly $d$-polytope, where $k=2s$ for some integer $s$. We say that a family $\mathcal{F}$ of $(s-1)$-faces of $P$ is {\em good} if every two elements $F\neq G$ of $\mathcal{F}$ satisfy the following conditions: $F$ and $G$ share at most $s/2$ vertices, while $F$ and $-G$ have no common vertices. To obtain an upper bound on the size of the vertex set $V$ of $P$ in terms of $d$ and $k$, one first observes that if $\mathcal{F}$ is a good family, $F\neq G$ are in  $\mathcal{F}$, and $b_F$ and $b_G$ are the barycenters of $F$ and $G$, then the polytopes $P+2b_F$  and $P+2b_G$ have disjoint interiors (cf.~Lemma \ref{disj-int}). One then uses a simple counting argument to show that there is a relatively large (in terms of $d$, $s$, and $|V|$) good family $\mathcal{F}$. Since the translates $\{P+2b_F \ : \  F\in \mathcal{F}\}$ of $P$ have pairwise disjoint interiors and are all contained in $3P$, the volume trick yields a desired upper bound on $|\mathcal{F}|$, and hence also on $|V|$; see the proof of Theorem 1.1 in \cite{LinNov} for more details.

The proof of the lower bound in Theorem \ref{cs-neighb-thm} is based on studying the cs transforms of cs polytopes introduced in \cite{McMShep} and on ``high-dimensional paradoxes" such as Ka\v{s}in's theorem \cite{Kasin} and its generalization due to Garnaev and Gluskin \cite{GarnGlus}. Specifically, Ka\v{s}in's theorem asserts that there is an absolute constant $C$ (for instance, $32$ does the job, see \cite[Lecture 4]{Ball}) such that for every $d$, $\R^{2d}$ has a $d$-dimensional subspace, $L_d$, with the following property: the ratio of the $\ell^2$-norm to the $\ell^1$-norm of any nonzero vector $x\in L_d$ is in the interval $[\frac{1}{\sqrt{2d}}, \frac{C}{\sqrt{2d}}]$; we refer to such a subspace as a Ka\v{s}in subspace. Via cs transforms, $d$-dimensional subspaces of $\R^{2d}$ correspond to cs $d$-polytopes with $4d$ vertices; furthermore, a careful analysis of cs transforms shows that the polytopes corresponding to Ka\v{s}in subspaces are $k$-neighborly with $k$ proportional to $d$. 

More generally, a theorem due to Garnaev and Gluskin \cite{GarnGlus} quantifies the extent to which an $n$-dimensional subspace of $\R^{n+d}$ can be ``almost Euclidean" (meaning that the ratio of the $\ell^2$-norm to the $\ell^1$-norm of non-zero vectors remains within certain bounds, more precisely, it is $\leq  \tilde{C}\sqrt{\frac{1+\log((n+d)/d)}{d}}$ for some absolute constant $\tilde{C}$). Via cs transforms, $n$-dimensional subspaces of $\R^{n+d}$ correspond to cs $d$-polytopes with $2(n+d)$ vertices, and the ``almost Euclidean" subspaces correspond to cs polytopes with neighborliness given by the lower bound in Theorem  \ref{cs-neighb-thm}, see \cite{LinNov} for technical details.  

The proof of the Garnaev--Gluskin result and hence  also of the lower bound part of Theorem \ref{cs-neighb-thm}  is probabilistic in nature: it does not give an explicit construction of neighborly cs polytopes, but rather shows that they form a set of positive probability in a certain probability space. Indeed it is an interesting open question to find an explicit construction for highly neighborly cs polytopes that meet the lower bound in Theorem \ref{cs-neighb-thm}. We discuss some known explicit constructions (for instance that of a cs $2$-neighborly $d$ polytope with $\approx \sqrt{3}^d$ vertices) in the next section. It would also be extremely interesting to shed some light on the exact values of $k(d,n)$: 

\begin{problem} Determine the exact values of $k(d,n)$, or at least find the value of  $n_0(d):=\min\{n : k(d,n)=1\}$, that is, find the number $n$ starting from which a cs $d$-polytope with $2(d+n)$ vertices cannot be even $2$-neighborly.
\end{problem}


\section{Towards an upper bound theorem for cs polytopes}
The fame of the cyclic polytope comes from the Upper Bound Theorem (UBT, for short) conjectured by Motzkin \cite{Motz57} and proved by McMullen \cite{McMullen70}. It asserts that among all $d$-polytopes with $n$ vertices, the cyclic polytope $C_d(n)$ simultaneously maximizes all the face numbers. The goal of this section is to summarize several upper-bound type results and problems for cs polytopes motivated by the UBT.


What is the maximum number of $k$-faces that a {\em centrally symmetric}  $d$-polytope with $N$ vertices can have? While our discussion in the previous section suggests that at present we are very far from even posing a plausible conjecture, certain asymptotic results on the maximum possible number of edges are known. Specifically, the following generalization of Theorem \ref{2-neighb} holds. This result is in sharp contrast with the fact that $f_1(C_d(n))={n \choose 2}$ as long as $d\geq 4$.

\begin{theorem}  \label{max-edges}
Let $d\geq 4$. If $P\subset \R^d$ is a cs $d$-polytope on $N$ vertices, then 
$$f_1(P) \leq (1-2^{-d})\frac{N^2}{2}.$$ 
On the other hand, there exist cs $d$-polytopes with $N$ vertices (for an arbitrarily large $N$) and at least 
$\left(1-3^{-\lfloor d/2-1\rfloor}\right){N \choose 2} \approx (1-\sqrt{3}^{-d})\cdot\frac{N^2}{2}$ edges. 
\end{theorem}

The first part of Theorem \ref{max-edges} was established by Barvinok and Novik in \cite[Proposition 2.1]{BarvN}; its proof relies on an extension of the argument discussed in the previous section to obtain an upper bound on the number of vertices that a cs $2$-neighborly $d$-polytope can have, and more specifically on Lemma \ref{disj-int}. The idea, very roughly, is as follows: each of the $N$ vertex-translates of $P$,  $P_u$ for $u\in V$, has the same volume as $P$, and all of them are contained in the polytope $2P$, whose volume is $2^d\vol(P)$. Hence, on average, an interior point of $2P$ belongs to $N/2^d$ (out of $N$) sets $\Int(P_u)$. Therefore, 
on average, $\Int(P_u)$ intersects with the interiors of at least $N/2^d-1$ other vertex-translates of $P$. Consequently, by Lemma \ref{disj-int}, the average degree of a vertex of $P$ in the graph of $P$ is $\leq N(1-2^{-d})$. This yields the desired upper bound on the number of edges of $P$.

The second part of Theorem \ref{max-edges} is due to Barvinok, Lee, and Novik \cite{BarvLeeN-antipodal}. The proof is based on an explicit construction whose origins can be traced to work of Smilansky  \cite{Smilansky}. 
 To discuss this part, we start by recalling that the cyclic polytope $C_d(n)$ is defined as the convex hull of $n$ points on the moment curve or, if $d=2k$ is even, on the trigonometric moment curve $\T_k: \R\to \R^{2k}$, where $$\T_k(t)=(\cos t, \sin t, \cos 2t, \sin 2t,\ldots, \cos kt, \sin kt).$$

In an attempt to come up with a cs analog of the cyclic polytope, Smilansky  \cite{Smilansky} (in the case of $k=2$), and Barvinok and Novik \cite{BarvN} (in the case of arbitrary $k$) considered the {\em symmetric moment curve}, $U_k : \R\to \R^{2k}$, defined by $$U_k(t)=\left(\cos t, \sin t, \cos 3t, \sin3t, \ldots, \cos (2k-1)t, \sin(2k-1)t \right).$$ Since $U_k(t)=U_k(t + 2\pi)$ for all $t\in\R$, from this point on, we think of $U_k$ as defined on the unit circle $\mathbb{S} =\R /2 \pi \Z.$ The name {\em symmetric moment curve} is explained by an observation that for all $t \in \mathbb{S}$, $t$ and $t+\pi$ form a pair of opposite points and $U_k(t+\pi)=-U_k(t)$. A \emph{bicyclic polytope} is then defined as the convex hull of $\{U_k(t) : t \in X\}$, where $X\subset \mathbb{S}$ is a finite subset of $\mathbb{S}$; we will also assume that $X$ is a \emph{cs} subset of $\mathbb{S}$. 

The papers \cite{Smilansky} (in the case of $k=2$),  and \cite{BarvN} along with \cite{Vinzant} (in the case of arbitrary $k$) study the edges of bicyclic polytopes. In particular, when $k=2$, the following result established in \cite{Smilansky} (see also \cite{BarvN}) holds. Recall that a face of a convex body $K$ is the intersection of $K$ with a supporting hyperplane.

\begin{theorem} \label{edge-sym-mom}
Let $\Gamma\subset\mathbb{S}$ be an open arc of length $2\pi/3$ and let $t_1,t_2\in \Gamma$. Then the line segment $\conv(U_2(t_1), U_2(t_2))$ is an edge of the $4$-dimensional convex body $\B_2:=\conv(U_2(t) : t\in\mathbb{S})$. 
\end{theorem}

One immediate consequence is
\begin{corollary} \label{base-case}
Let $X=\{0, \pi/2, \pi, 3\pi/2\}\subset\mathbb{S}$, let $s\geq 2$ be an integer,  and let $X_s$ be a cs subset of $\mathbb{S}$ obtained from $X$ by replacing each $\tau\in X$ with a cluster of $s$ points all of which lie  on a small arc containing $\tau$. Then the polytope $Q_s:=\conv(U_2(t) : t\in X_s)$ is a cs $4$-polytope that has $N:=4s$ vertices and at least $\frac{1}{2}\cdot N(\frac{3}{4}N-1)\approx \frac{3}{4}{N \choose 2}$ edges.
\end{corollary}
\noindent Indeed, it follows from Theorem \ref{edge-sym-mom} that each vertex of $Q_s$ is connected by an edge to all other vertices of $Q_s$ except possibly those coming from the opposite cluster, yielding the result.

Denote by $\fmax(d,N; k-1)$ the maximum possible number of $(k-1)$-faces that a cs $d$-polytope with $n$ vertices can have. From the above discussion we infer that
\[
\frac{3}{4}\cdot \frac{N^2}{2}-O(n) \leq \fmax(4,N; 1) \leq \frac{15}{16}\cdot \frac{N^2}{2}.
\]
These are currently the \emph{best} known bounds on the maximum possible number of edges that a cs $4$-polytope with $N$ vertices can have.

Perhaps somewhat surprisingly, for $k>2$, bicyclic polytopes do not have a record number of edges. However, the symmetric moment curve is used in the following construction that produces cs polytopes with the largest known to-date number of edges.

Let $m\geq 1$ be an integer. Define $\Phi_m(t) : \mathbb{S}\to \R^{2m+2}$ by \begin{eqnarray*}
\Phi_m(t)&:=&(U_1(t), U_1(3t), U_1(3^2t),\ldots, U_1(3^m(t))\\
&=&(\cos t, \sin t, \cos 3t, \sin 3t, \cos 9t, \sin 9t,\ldots, \cos(3^{m}t), \sin(3^{m}t)).
\end{eqnarray*}

Parts 1 and 2 of the following result complete the proofs of Theorem \ref{2-neighb} and Theorem \ref{max-edges}, respectively.

\begin{theorem}  \label{record-edges}
Fix integers $m\geq 2$ and $s\geq 2$. Let $A_m$ be a set of $2(3^{m}-1)$ equally spaced points on $\mathbb{S}$, and let $A_{m,s}$ be a cs subset of $\mathbb{S}$ obtained from $A_m$ by replacing each $\tau\in A_m$ with a cluster of $s$ points, all of which lie on a very small arc containing $\tau$. Then
\begin{enumerate}
\item the polytope $P_m:=\conv(\Phi_m(A_m))$ is a cs
$2$-neighborly polytope of dimension $2(m+1)$ that has
$2(3^{m}-1)$ vertices, 
\item the polytope $P_{m,s}:=\conv(\Phi_m(A_{m,s}))$
is a cs polytope of dimension $2(m+1)$
that has $N:=2s(3^{m}-1)$ vertices and more than
$\left(1-3^{-m}\right)\binom{N}{2}$ edges.
\end{enumerate}
  \end{theorem}

The assumption $m\geq 2$ is only needed to guarantee that the dimension of $P_m$ is exactly $2(m+1)$ rather than $\leq 2(m+1)$. 
To prove that $P_m$ is cs $2$-neighborly and $P_{m,s}$ has many edges for all $m\geq 1$, it is enough to show that each vertex of $P_{m,s}$ is connected by an edge to all other vertices of $P_{m,s}$ except possibly those coming from the opposite cluster. This can be done by induction on $m$. Since $\Phi_1=U_2$, the case of $m=1$ is simply Corollary \ref{base-case}. For the inductive step, one relies on some standard facts about faces of polytopes along with
an observation that the composition of $\Phi_m : \R\to \R^{2(m+1)}$ with the projection of $\R^{2(m+1)}$ onto $\R^{2m}$ that forgets the first two coordinates is the curve $t\to \Phi_{m-1}(3t)$, while the composition of $\Phi_m$ with the projection of $\R^{2(m+1)}$ onto $\R^4$ that forgets all but the first four coordinates is the curve $\Phi_1$; see \cite[Section 3]{BarvLeeN-antipodal} for more details. Finally, to extend the construction of Theorem \ref{record-edges} to {\em odd} dimensions, consider the bipyramid over polytopes $P_m$ and $P_{m,s}$. 

 \smallskip\noindent To summarize, 
\begin{equation} \label{fmax_1}
\left(1-3^{-\lfloor d/2 -1\rfloor}\right){N \choose 2} \leq \fmax(d,N;1)\leq \left(1-2^{-d}\right)\cdot\frac{N^2}{2}.
\end{equation}

To extend the above discussion to higher-dimensional faces, we need to look more closely at the curve $U_k$ and its convex hull.

One crucial feature of the convex hull $\M_{k} =\conv(\T_k(t):  t\in\mathbb{S}) \subset \R^{2k}$ of the trigonometric moment curve is that it is $k$-neighborly, that is, for any  $p \leq k$ distinct points $t_1, \ldots, t_p \in \mathbb{S}$, the convex hull $\conv(\T_k(t_1), \ldots, \T_k(t_n))$ is a face of $\M_k$; see, for example, Chapter II of \cite{Ba02}. While the convex hull of $U_k$ is not $k$-neighborly, the following theorem, which is the main result of \cite{BarvLeeN-neighborly}, shows that it is {\em locally $k$-neighborly} (cf.~Theorem \ref{edge-sym-mom}).

\begin{theorem} \label{neighb-symm-mom}
For every positive integer $k$ there exists a number 
${\pi \over 2}  < \alpha_k  < \pi$
such that for an arbitrary open arc $\Gamma \subset \mathbb{S}$ of length $\alpha_k$ 
and arbitrary distinct $p \leq k$ points $t_1, \ldots, t_p \in \Gamma$, the set 
$\conv(U_k(t_1), \ldots, U_k(t_p))$
is a face of $\B_k:=\conv(U_k(t) : t\in\mathbb{S})$. 
\end{theorem}


The gap between the currently known upper and lower bounds on $\fmax(d,N; k-1)$ for $k>2$ is much worse than the gap for $k=2$   illustrated by eq.~\eqref{fmax_1}. Indeed, we have: 
\begin{theorem} \label{k-faces} Let $3\leq k \leq d/2$. Then 
\[ \left(1-k^2\left(2^{3/20k^2 2^k}\right)^{-d}\right) {N \choose k}
\leq \fmax(d,N; k-1) \leq \left(1-2^{-d}\right) \frac{N}{N-1}{N \choose k}.
\]
\end{theorem}
The proof of the upper bound part of this theorem follows easily from the first part of Theorem \ref{max-edges} together with (1) the well-known perturbation trick that reduces the situation to cs \emph{simplicial} polytopes, and (2) the standard double-counting argument that relates the number of edges to the number of $(k-1)$-faces in any simplicial complex with $N$ vertices, see \cite[Proposition 2.2]{BarvN} for more details. 

The proof of the lower bound part can be found in \cite[Section 4]{BarvLeeN-antipodal}. It utilizes a construction that is somewhat along the lines of the construction of Theorem \ref{record-edges}, but quite a bit more involved: the desired polytope is obtained as the convex hull of a carefully chosen set of points on the curve $\Psi_{k,m}: \mathbb{S}\to \R^{2k(m+1)}$
 defined by \[\Psi_{k,m}(t):=(U_k(t),U_k(3t), U_k(3^2t),\ldots, U_k(3^m t)).\]  (Note that $\Phi_m$ is essentially $\Psi_{2,m}$: in $\Psi_{2,m}(t)$ every coordinate except for the first two and the last two shows up twice; to obtain $\Phi_m$ from $\Psi_{2,m}$ simply leave one copy of each repeated coordinate.) 
To choose an appropriate set of points one uses the notion of a $k$-independent family of subsets of $\{1,2,\ldots,m\}$ and a deterministic construction from \cite{FrLiLe} of a large $k$-independent family. Finally, to show that the resulting polytope has many $(k-1)$-faces, one relies heavily on the local neighborliness of the convex hull of $U_k$ discussed in Theorem \ref{neighb-symm-mom} along with some standard results on faces of polytopes.

As the discussion above indicates, at present our understanding of the maximum possible number of faces of a given dimension that a cs $d$-polytope with a given number of vertices can have is rather limited even for $d=4$. In particular, the following questions are wide open:
\begin{problem} \label{face-numbers}
Does the limit  $\lim_{N\rightarrow\infty} \frac{\fmax(d,N;1)}{{N \choose 2}}$ exist and, if so, what is its value? Or better yet, what is the actual value of $\fmax(d,N;1)$? Similarly, what is
 $\lim_{N\rightarrow\infty} \frac{\fmax(d,N;k-1)}{{N \choose k}}$ for $2<k\leq d/2$?
Can we, at least, establish better lower and upper bounds than those given in Theorem~\ref{k-faces}?
\end{problem}
In fact, it should be stressed that for $d\geq 6$ and $N\geq 2(d+2)$, we do not even know if in the class of cs $d$-polytopes with $N$ vertices there is a polytope that simultaneously maximizes all the face numbers (For $d=4,5$, a cs simplicial polytope that maximizes the number of edges automatically maximizes the rest of face numbers).

\section{Towards an upper bound theorem for cs simplicial spheres}
McMullen's Upper Bound Theorem was extended by Stanley \cite{Stanley75} to the class of all simplicial spheres. More precisely, Stanley proved that among all 
simplicial $(d-1)$-spheres with $n$ vertices, the boundary complex of the cyclic {\em polytope} $C_d(n)$ simultaneously maximizes all the face numbers. (This result was extended even further to some classes of simplicial manifolds and even certain pseudomanifolds with isolated singularities, see \cite{N98,Hersh-Nov,NSw-isol}. In fact, already in 1964, Victor Klee \cite{Klee-UBT} proved that the UBT holds for all $(d-1)$-dimensional {\em Eulerian} complexes that have at least $O(d^2)$ vertices. Whether the UBT holds for all Eulerian complexes remains an open question.)

The situation with the face numbers of cs spheres versus face numbers of cs polytopes appears to be drastically different.
On one hand, as we saw in Sections 3 and 4, a cs $4$-polytope $P$ with  $N\geq 12$ vertices cannot be $2$-neighborly; furthermore, such a $P$ has at most $\frac{15}{16}\cdot \frac{N^2}{2}$ edges. On the other hand, it is a result of Jockusch \cite{Jockusch95} that


\begin{theorem}  \label{cs-2-neighb-spheres}
 For every even number $N\geq 8$, there exists a cs $2$-neighborly 
simplicial sphere $J_{N}$ of dimension $3$ with $N$ vertices.  In particular, $J_{N}$ has $\binom{N}{2} - \frac{N}{2}=\frac{N^2}{2}-N$ edges.
\end{theorem} 

Jockusch's proof, see \cite{Jockusch95}, is by inductive construction. The initial sphere, $J_8$, is the boundary complex of the cross-polytope $\C^*_4$ with the involution $\phi$ induced by the map $\phi(v)=-v$ on the vertex set of $\C^*_4$. The cs sphere $(J_{N+2}, \phi)$ is obtained from the cs sphere $(J_N, \phi)$ by the following procedure: 
\begin{enumerate}
\item Find a subcomplex $B_N$ in $J_N$ such that (i) $B_N$ is a $3$-ball, (ii) $B_N$ and $\phi(B_N)$ share no common facets,  (iii) every vertex of $J_N$ is a vertex of $B_N$, and (iv) every edge of $B_N$ lies on the {\em boundary} of $B_N$, i.e., $B_N$ has no interior vertices and no interior edges. 
\item Let $v_{N+1}$ and $v_{N+2}$ be two new vertices. Cut out the interior of $B_N$ from $J_N$, and cone the boundary of the resulting hole with $v_{N+1}$. Similarly, cut out the interior of $\phi(B_N)$ from $J_N$, and cone the boundary of the resulting hole with $v_{N+2}$.
\end{enumerate}
Extending $\phi$ to $J_{N+2}$ by letting $\phi(v_{N+1})=v_{N+2}$ and $\phi(v_{N+2})=v_{N+1}$ provides us with a free involution on $J_{N+2}$. Furthermore, choosing $B_N$ in a way that is specified in Part (1) of the construction guarantees that the cs sphere $J_{N+2}$ is $2$-neighborly (provided $J_N$ was $2$-neighborly). To allow for this inductive construction, an essential part of Jockusch's proof is devoted to defining $B_N$ in a way that ensures that the resulting simplicial sphere $J_{N+2}$ has a subcomplex $B_{N+2}$ with the same properties. 

The {\em suspension} of a simplicial complex $\Delta$, $\Sigma(\Delta)$, is the join of $\Delta$ with the $0$-dimensional sphere, that is, $\Sigma(\Delta)= \{F, F\cup \{u_0\}, F\cup \{w_0\} \, : \, F\in\Delta\}$, where $u_0$ and $w_0$ are two new vertices. Observe that if $\Delta$ is cs, then $\Sigma(\Delta)$ is cs (with $u_0$ and $w_0$ being antipodes), and if $\Delta$ is a $(d-1)$-sphere, then $\Sigma(\Delta)$ is a $d$-sphere; furthermore, if $\Delta$ is cs $k$-neighborly, then so is $\Sigma(\Delta)$. In particular, for every even $N\geq 8$, $\Sigma(J_N)$ is a cs $2$-neighborly 
simplicial $4$-sphere with $N+2$ vertices. Jockusch's results lead to the following problem on higher-dimensional cs spheres.

\begin{problem} \label{cs-neighb-sph}
Let $d>5$ and let $N\geq 2d$ be an even number. Is there a cs $\lfloor d/2\rfloor$-neighborly 
simplicial $(d-1)$-sphere with $N$ vertices?
\end{problem}

The boundary complex of $\C^*_d$ is the unique cs $d$-neighborly simplicial $(d-1)$-sphere with $2d$ vertices. The boundary complex of $\conv(\pm e_1,\ldots,\pm e_d, \pm\sum_{i=1}^d e_i)$ is a cs $\lfloor d/2\rfloor$-neighborly simplicial $(d-1)$-sphere with $2(d+1)$ vertices \cite{McMShep}; furthermore, in the case of an odd $d$, \cite[Section 6.2]{Bier-spheres} provides a construction of many cs $\lfloor d/2\rfloor$-neighborly simplicial $(d-1)$-spheres with $2(d+1)$ vertices. Lutz \cite[Chapter 4]{Lutz} found (by a computer search) several cs $3$-neighborly simplicial $5$-spheres with $16=2(6+2)$ vertices; his examples have vertex-transitive cyclic symmetry. Suspensions of Lutz's examples are cs $3$-neighborly simplicial $6$-spheres with $18$ vertices. For all other values of $d$ and $N$, Problem \ref{cs-neighb-sph} remains wide open. It is also worth pointing out that Pfeifle \cite[Chapter 10]{Pfeifle} investigated possible neigborliness of cs star-shaped spheres --- a class of objects that contains all boundary complexes of cs simplicial polytopes and is contained in the class of all cs simplicial spheres. In analogy with McMullen--Shephard's result \cite{McMShep} about cs $d$-polytopes with $2(d+2)$ vertices, he proved that for all even $d \geq 4$ and for all odd $d \geq 11$, there are no cs $\lfloor d/2\rfloor$-neighborly star-shaped spheres of dimension $d-1$ with $2(d+2)$ vertices.

One of the main reasons for trying to resolve Problem \ref{cs-neighb-sph} comes from the following cs analog of the Upper Bound Theorem; this result is due to Adin \cite{Adin91} and Stanley (unpublished).

\begin{theorem} \label{cs-spheres-UBT}
Among all cs  simplicial $(d-1)$-spheres with $N$ vertices, a cs $\lfloor d/2\rfloor$-neighborly $(d-1)$-sphere with $N$ vertices simultaneously maximizes all the face numbers, assuming such a sphere exists.
\end{theorem}

We sketch the proof of Theorem \ref{cs-spheres-UBT} in the next section; as in the classic non-cs case, it relies on the theory of Stanley--Reisner rings and on the Dehn--Sommerville relations. 

To close this section, we posit a weaker, and hence potentially more approachable version, of Problem \ref{cs-neighb-sph}. Let $\Delta$ be a $(d-1)$-dimensional simplicial complex and $F$ a face of $\Delta$. The {\em link} of $F$ in $\Delta$ is $\lk_\Delta(F):=\{G\in\Delta \ : \ F\cup G\in \Delta, \, F\cap G=\emptyset\}$. (Thus, the link of $F$ describes the local structure of $\Delta$ around $F$.)  We say that $\Delta$ is {\em pure} if all facets of $\Delta$ have dimension $d-1$; equivalently, for any face $F\in\Delta$, the link of $F$ is $(d-|F|-1)$-dimensional. Furthermore, we say that $\Delta$ is {\em Eulerian} if it is pure and the link of any face $F\in\Delta$ (including the empty face) has the same Euler characteristic as ${\mathbb S}^{d-|F|-1}$ --- the $(d-|F|-1)$-dimensional sphere.  In particular, the class of Eulerian complexes includes all simplicial spheres, all odd-dimensional simplicial manifolds as well as all even-dimensional manifolds whose Euler characteristic is two. Eulerian complexes were introduced by Victor Klee in \cite{Klee64}. 

\begin{problem}
Let $d>5$ and let $N\geq 2d$ be an even number. Is there a cs $\lfloor d/2\rfloor$-neighborly $(d-1)$-dimensional Eulerian complex with $N$ vertices?
\end{problem}

\section{The algebraic side of the story: Stanley--Reisner ring}
We now switch to the algebraic side of the story and present a quick review of the major algebraic tool in the study of face numbers of simplicial complexes --- the Stanley--Reisner ring. Along the way, we 
outline the proof of Theorem \ref{cs-spheres-UBT} as well as prepare the ground for our discussion of the lower bound theorems in the next section. The main reference to this material is Stanley's book \cite{Stanley96}.

Let $\Delta$ be a simplicial complex with vertex set $V$ and let $\field$ be an {\em infinite} field of an arbitrary characteristic.  
Consider a polynomial ring $S:=\field[x_v \, : \, v\in V]$ with one variable for each vertex in $\Delta$.  The \emph{Stanley--Reisner ideal} of $\Delta$ is the following squarefree monomial ideal $$I_{\Delta}:= \left(x_{v_1}x_{v_2}\cdots x_{v_k} \ : \ \{v_1,v_2,\ldots,v_k\} \notin \Delta \right).$$ The \emph{Stanley--Reisner ring} (or face ring) of $\Delta$ is the quotient $\field[\Delta]:= S/I_{\Delta}$. Since $I_{\Delta}$ is a monomial ideal, the quotient ring $\field[\Delta]$ is graded by degree.  The definition of $I_{\Delta}$ ensures that, as a $\field$-vector space, each graded piece of $\field[\Delta]$, denoted $\field[\Delta]_i$, has a natural basis of monomials whose supports correspond to faces of $\Delta$.

Stanley's and Hochster's insight (independently from each other) in defining this ring \cite{Stanley75,Hochster} was that algebraic properties of $\field[\Delta]$ reflect many combinatorial and topological properties of $\Delta$. For instance, if $\Delta$ is $(d-1)$-dimensional then the Krull dimension of $\field[\Delta]$ is $d$; in fact, the Hilbert series of $\field[\Delta]$, i.e., $\Hilb(\field[\Delta], t):=\sum_{i=0}^{\infty} \dim_\field \field[\Delta]_i \cdot t^i$,
is given by 
$$ \Hilb(\field[\Delta], t)=\sum_{i=0}^d \frac{f_{i-1}(\Delta) t^i}{(1-t)^i}=\frac{\sum_{i=0}^d f_{i-1}(\Delta) t^i(1-t)^{d-i}}{(1-t)^d}.$$

The last equation leads to the following definition: if $\Delta$ is a $(d-1)$-dimensional simplicial complex, then the {\em $h$-vector} of $\Delta$, $h(\Delta)=(h_0(\Delta),h_1(\Delta),\ldots,h_d(\Delta))$, is the vector whose entries satisfy $\sum_{i=0}^d h_i(\Delta)t^i=\sum_{i=0}^d f_{i-1}(\Delta) t^i(1-t)^{d-i}$; equivalently,  
\begin{equation} \label{h-numbers}
\sum_{i=0}^d h_i(\Delta)t^{d-i}=\sum_{i=0}^d f_{i-1}(\Delta) (t-1)^{d-i}.
\end{equation}
In particular, $h_0(\Delta)=1$, $h_1(\Delta)=f_0(\Delta)-d$, and $h_2(\Delta)=f_1(\Delta)-(d-1)f_0(\Delta)+{d \choose 2}$.

The following immediate consequences of eq.~\eqref{h-numbers} are worth mentioning:  knowing the $f$-numbers of $\Delta$ is equivalent to knowing its $h$-numbers. Moreover, since the $f$-numbers are \textit{nonnegative} integer combinations of the $h$-numbers, any upper/lower bounds on the $h$-numbers of $\Delta$ automatically imply upper/lower bounds on the $f$-numbers of $\Delta$.  

Let $\Delta$ be a $(d-1)$-dimensional simplicial complex. A sequence of linear forms, $\theta_1, \theta_2, \ldots, \theta_d \in S$ is called a {\em linear system of parameters} (or l.s.o.p.) for $\field[\Delta]$ if the ring $\field[\Delta]/(\Theta)$ is a finite-dimensional $\field$-vector space; here $(\Theta):=(\theta_1,\ldots,\theta_d)$. It is well-known that if $\field$ is an infinite field, then $\field[\Delta]$ admits an l.s.o.p. 
An l.s.o.p.~for $\field[\Delta]$ is a {\em regular sequence} if $\theta_{i+1}$ is a non-zero-divisor on $\field[\Delta]/(\theta_1,\ldots,\theta_i)$ for all $0\leq i<d$. We say that $\Delta$ is {\em Cohen--Macaulay} (over $\field$), or {\em $\field$-CM} for short, if every l.s.o.p.~for $\field[\Delta]$ is a regular sequence.

Assume now that $\Delta$ is $\field$-CM, and $\theta_1, \ldots, \theta_d$ is an l.s.o.p.~for $\field[\Delta]$. Then $\theta_1$ is a non-zero-divisor on $\field[\Delta]$, and so the following sequence of $\field$-vector spaces
\begin{equation} \label{exact-seq}
0\to \field[\Delta]_{i-1}\stackrel{\cdot\theta_1}{\longrightarrow}\field[\Delta]_{i} \longrightarrow \big(\field[\Delta]/(\theta_1)\big)_i\to 0
\end{equation}
is exact (for all $i\geq 0$). Thus, $\dim_\field (\field[\Delta]/(\theta_1))_i =\dim_\field \field[\Delta]_i-\dim_\field \field[\Delta]_{i-1}$ for all $i$. Multiplying by $t^i$ and summing up the resulting equations, we obtain that $\Hilb(\field[\Delta]/(\theta_1), t)=(1-t)\Hilb(\field[\Delta], t)$. Iterating this argument for $\theta_2,\ldots,\theta_d$ leads to the following result due to Stanley \cite[Section 4]{Stanley75}. 

\begin{theorem} \label{h-numbers-CM} 
 Let $\Delta$ be a $(d-1)$-dimensional $\field$-CM complex and let $\theta_1, \ldots, \theta_d$ be an l.s.o.p.~for $\field[\Delta]$. Then $\Hilb(\field[\Delta]/(\Theta),t) = (1-t)^d \Hilb(\field[\Delta],t)=\sum_{i=0}^d h_i(\Delta) t^i$. (In particular, the $h$-numbers of CM complexes are non-negative.)
\end{theorem}

If $\Delta\subseteq \Gamma$ are simplicial complexes (say, on the same vertex set $V$), then $I_\Delta\supseteq I_\Gamma$, and so there is a natural surjection $\field[\Gamma]\to \field[\Delta]$. Furthermore, if $\dim\Delta=\dim\Gamma$, then any l.s.o.p.~$\theta_1,\ldots,\theta_d$ for  $\field[\Gamma]$ is also an l.s.o.p.~for $\field[\Delta]$, and the induced graded homomorphism 
$\field[\Gamma]/(\Theta)\to \field[\Delta]/(\Theta)$ is surjective. This observation together with Theorem \ref{h-numbers-CM} yields the following special case of \cite[Theorem 2.1]{Stanley-93}:

\begin{corollary} \label{h-subcomplexes}
Let $\Delta\subseteq \Gamma$ be simplicial complexes. Assume that both $\Delta$ and $\Gamma$ are $\field$-CM and have the same dimension. Then $h_i(\Delta)\leq h_i(\Gamma)$ for all $i$.
\end{corollary} 

The reason CM complexes are relevant to our discussion is that, by a result of Reisner \cite{Reisner}, $\Delta$ is a $\field$-CM complex if and only if $\Delta$ is pure and for every face $F$ of $\Delta$ (including the empty face), all but the top homology group (computed over $\field$) of the link of $F$ vanish. In particular, all simplicial spheres and balls are CM over any field; furthermore, the $j$-skeleton of a $\field$-CM complex is $\field$-CM for all $j$.

Another very important result about simplicial spheres is Dehn--Sommerville relations established by Klee \cite{Klee64}: if $\Delta$ is a simplicial $(d-1)$-sphere, then $h_i(\Delta)=h_{d-i}(\Delta)$ for all $0\leq i \leq d$. (In fact, Klee showed that these relations hold for all Eulerian complexes.) 

We are now ready to close this section with a sketch of the proof of Stanley's UBT and of its cs analog -- Theorem \ref{cs-spheres-UBT}. Following the custom, if $P$ is a simplicial polytope, we denote by $h(P)$ the $h$-vector of the boundary complex of $P$.

Let $\Delta$ be a simplicial $(d-1)$-sphere with vertex set $V$, $|V|=n$. Let $\overline{V}$ be the simplex on $V$, let $\Gamma=\Skel_{d-1}(\overline{V})$, and let $C_d(n)$ be the cyclic polytope. Then $\Delta\subseteq \Gamma$ are both CM complexes of the same dimension. Hence, by Corollary \ref{h-subcomplexes}, $h_i(\Delta)\leq h_i(\Gamma)$ for all $0\leq i \leq d$. Furthermore, since $C_d(n)$ is $\lfloor d/2\rfloor$-neighborly, $\partial C_d(n)$ and $\Gamma$ have the same $(\lfloor d/2\rfloor-1)$-skeleton, and so $h_i(\Gamma)=h_i(C_d(n))$ for all $0\leq i \leq d/2$. (Indeed, $h_i$ is determined by $f_0, f_1,\ldots, f_{i-1}$.) We conclude that 
 $h_i(\Delta)\leq h_i(C_d(n))$ for all $0\leq i \leq d/2$.  
Dehn--Sommerville relations, applied to $\Delta$ and $\partial C_d(n)$, then yield that  $h_i(\Delta)\leq h_i(C_d(n))$ also holds for all $d/2 \leq i \leq d$. Thus, $h_i(\Delta)\leq h_i(C_d(n))$ for all $0\leq i \leq d$, and the Upper Bound Theorem, asserting that $f_{j}(\Delta)\leq f_{j}(C_d(n))$ for all $1\leq j \leq d-1$, follows.

Similarly, if $\Delta$ is a cs simplicial $(d-1)$-sphere with $N=2m$ vertices, then $\Delta$ is a full-dimensional subcomplex of $\Gamma:=\Skel_{d-1}(\partial\C^*_m)$ (under any bijection from the vertex set of $\Delta$ to that of $\partial\C^*_m$ that takes antipodal vertices to antipodal vertices). Further, if $S$ is a cs $\lfloor d/2\rfloor$-neighborly simplicial $(d-1)$-sphere with $N$ vertices, then $\Skel_{\lfloor d/2\rfloor-1}(\Gamma)= \Skel_{\lfloor d/2\rfloor-1}(S)$. Applying the same argument as above to $\Delta$, $\Gamma$, and $S$, we conclude that $h_i(\Delta)\leq h_i(S)$ for all $0\leq i \leq d$, and hence that $f_{j}(\Delta)\leq f_{j}(S)$ for all $1\leq j \leq d-1$. This completes the proof of Theorem \ref{cs-spheres-UBT}.

 Note that whether a cs $\lfloor d/2\rfloor$-neighborly simplicial $(d-1)$-sphere with $N=2m$ vertices exists or not, the $h$-numbers (and hence also $f$-numbers) such a sphere would have are well-defined: $h_i(S)=h_{d-i}(S)=h_i(\Skel_{d-1}(\C^*_m))$ for all $i\leq d/2$. Thus, independently of the existence of such a sphere, Theorem \ref{cs-spheres-UBT} provides upper bounds on face numbers of cs simplicial $(d-1)$-spheres with $N$ vertices. However, if a cs $\lfloor d/2\rfloor$-neighborly simplicial $(d-1)$-sphere with $N$ vertices exists, then these bounds are tight.

Note also that the proof of Theorem \ref{cs-spheres-UBT} almost does not use the central symmetry assumption: instead, it relies on a much weaker condition, namely, on the existence of a free involution $\phi: V(\Delta) \to V(\Delta)$ on the vertex set of $\Delta$ such that $\{v,\phi(v)\}$ is not an edge of $\Delta$ for all $v\in V(\Delta)$. We refer our readers to \cite{N05}  for more details on the $h$-vectors of simplicial spheres and manifolds possessing this weaker property and to \cite{BrowderN} for a complete characterization of $h$-vectors of CM complexes with this property.

\section{The generalized lower bound theorem for cs polytopes}
Our ultimate dream is to find a cs analog of the $g$-theorem for cs simplicial polytopes. While at the moment it is completely out of reach (we do not even have a plausible upper bound conjecture, let alone a complete characterization!), establishing the lower-bound type results is a necessary part of the program. To this end, in this section we discuss a cs analog of the Generalized Lower Bound Theorem for cs simplicial polytopes. We start by reviewing  the classical Lower Bound Theorem (LBT, for short) and the Generalized Lower Bound Theorem (GLBT, for short). To state these results, we need a few definitions. 

A {\em triangulation} of a simplicial $d$-polytope $P$ is a simplicial $d$-ball $B$ whose boundary, $\partial B$, coincides with $\partial P$. A simplicial $d$-polytope $P$ is called {\em $(r-1)$-stacked} (for some $1\leq r \leq d$) if there exists a triangulation $B$ of $P$ such that $\Skel_{d-r}(B)=\Skel_{d-r}(\partial P)$, i.e., all ``new'' faces of this triangulation have dimension $>d-r$. Note that the simplices are the only $0$-stacked polytopes, and that $1$-stacked polytopes --- usually referred to as {\em stacked} polytopes --- are polytopes that can be obtained from the $d$-simplex by repeatedly attaching (shallow) $d$-simplices along facets.  In contrast with the cyclic polytopes, two stacked $d$-polytopes with $n$ vertices may not have the same combinatorial type. However, they do have the same face numbers.

\begin{theorem} {\rm (LBT)} \label{LBT} Let $P$ be a simplicial $d$-polytope with $d\geq 4$. Then 
$h_1(P)\leq h_2(P)$, with equality if and only if  $P$ is stacked.
\end{theorem}

\begin{theorem} {\rm (GLBT)} \label{GLBT} Let $P$ be a simplicial $d$-polytope. Then 
$$1=h_0(P)\leq h_1(P)\leq h_2(P)\leq \cdots\leq h_{\lfloor d/2\rfloor}(P).$$ 
Furthermore, $h_r(P)=h_{r-1}(P)$ for some $r\leq d/2$ if and only if $P$ is $(r-1)$-stacked.
\end{theorem}
 
Since $f_{i-1}$ is a non-negative linear combination of $h_0, h_1, \ldots, h_i$, one immediate corollary of Theorem \ref{LBT} is that among all simplicial $d$-polytopes with $n$ vertices, a stacked polytope has the smallest number of edges. In fact, Theorem \ref{LBT} together with a well-known reduction due to McMullen, Perles, and Walkup implies that among all simplicial $d$-polytopes with $n$ vertices, stacked polytopes simultaneously minimize \emph{all} the face numbers, and that for $d\geq 4$, stacked polytopes are the only minimizers, see \cite[Section 5]{Kalai87}.

The differences between consecutive $h$-numbers are known as the {\em $g$-numbers}: $g_0(P):=1$ and $g_r(P):=h_r(P)-h_{r-1}(P)$ for $1\leq r\leq d/2$. Non-negativity of $g_2$ for simplicial polytopes of dimension $4$ and higher was established by Barnette \cite{Barnette73},
while Billera and Lee \cite{BilleraLee} proved that the equality $g_2(P)=0$ (for $d\geq 4$) holds if and only if $P$ is stacked. The first part of the GLBT was proved by Stanley \cite{Stanley80} (a more elementary proof of this part is due to McMullen \cite{McMullen93, McMullen96}). The second part was conjectured by McMullen and Walkup, \cite{McMullenWalkup71}, and proved only recently by Murai and Nevo \cite{MuraiNevo2013}.

Since the $d$-dimensional cross-polytope has the minimal number of faces among all cs simplicial $d$-polytopes and since $h_r(\C^*_d)={d \choose r}>{d \choose r-1}=h_{r-1}(\C^*_d)$ for all $1\leq r\leq d/2$, one expects that for this class of polytopes, the inequalities $g_r(P)\geq 0$ of the GLBT can be considerably strengthened. Indeed, the following result of Stanley \cite{Stanley-87} provides a cs version of the inequality part of the GLBT. This result settled an unpublished conjecture by Bj\"orner.

\begin{theorem} \label{cs-GLBT-ineq} Let $P$ be a cs simplicial $d$-polytope. Then
$$
g_r(P)\geq {d \choose r} -{d \choose r-1}=g_r(\C^*_d) \quad \mbox{for all } r\leq d/2.
$$
\end{theorem}

Stanley's proof of the first part of Theorem \ref{GLBT} is based on the theory of toric varieties associated with polytopes. Here is a very rough sketch of the argument. 
If $P\subset \R^d$ is a simplicial $d$-polytope, then a slight perturbation of the vertices of $P$ does not change its combinatorial type, and so we may assume without loss of generality that $P$ has rational vertices; further, by translating $P$, we may also assume that the origin is in the interior of $P$. Let $V$ be the vertex set of $P$. For each $v\in V$, let $(v_1,v_2,\ldots,v_d)$ denote the coordinates of $v$ in $\R^d$, and define $\theta_j:=\sum_{v\in V} v_j x_v \in \R[x_v \ : \ v\in V]$ for $j=1,2,\ldots,d$. Consider $\R[\partial P]$ --- the Stanley--Reisner ring of the boundary complex of $P$ over $\R$. Then $\theta_1,\ldots,\theta_d$ is an l.s.o.p.~of $\R[\partial P]$ (this follows, for instance, from \cite[Theorem III.2.4]{Stanley96}), and so by Theorem \ref{h-numbers-CM}, $\dim_\R \left(\R[\partial P]/(\Theta)\right)_i=h_i(P)$. On the other hand, $\R[\partial P]/(\Theta)$ is isomorphic to the singular cohomology ring of the toric variety $X_P$ corresponding to $P$ (see \cite[Theorem 10.8]{Danilov} and \cite[Section 5.2]{Fulton} for more details). Let $\omega=\sum_{v\in V} x_v$. The toric variety $X_P$ is known to satisfy the hard Lefschetz theorem, which implies that for all $i\leq d/2$, the multiplication map $\cdot\omega: \left(\R[\partial P]/(\Theta)\right)_{i-1} \to \left(\R[\partial P]/(\Theta)\right)_i$ is injective, and hence that $h_{i-1}(P)\leq h_i(P)$ for all $i\leq d/2$, as desired.

 Extending Theorem \ref{GLBT} to all simplicial spheres is a major open problem in the field: it is a part of the celebrated $g$-conjecture. The main obstacle is that all known proofs of the $g$-theorem for simplicial polytopes, see \cite{Stanley80} and \cite{McMullen93,McMullen96}, rely heavily on such tools as toric varieties and the hard Lefschetz theorem or polytopal algebras and the Hodge--Riemann--Minkowski quadratic inequalities between mixed volumes, that is, tools that are (at least at present) available only for convex polytopes.

To prove Theorem \ref{cs-GLBT-ineq} a few additional ideas are needed. Let $P$ be a cs simplicial $d$-polytope or, more generally, let $(\Delta, \phi)$ be a cs CM complex over $\R$ with vertex set $V$. Then the involution $\phi$ induces the map $\sigma: x_v\to x_{\phi(v)}$ on the set of variables of $S=\R[x_v \ : \ v\in V]$, which, in turn, extends to a unique automorphism $\sigma$ on $S$. Furthermore, since $\sigma(I_\Delta)=I_\Delta$, the map $\sigma$ gives rise to an automorphism on  $\R[\Delta]$, which we also denote by $\sigma$. The main insight of \cite{Stanley-87} is that this allows us to equip $\R[\Delta]$ with a finer grading by $\N\times \Z/2\Z$: indeed, define 
\[\R[\Delta]_{(i,+1)}:=\{f\in\R[\Delta]_i \ : \ \sigma(f)=f\} \quad \mbox{and} \quad
 \R[\Delta]_{(i,-1)}:=\{f\in\R[\Delta]_i \ : \ \sigma(f)=-f\}.
\]
 Then $\R[\Delta]_i=\R[\Delta]_{(i,+1)}\oplus \R[\Delta]_{(i,-1)}$ as vector spaces while $\R[\Delta]_{(i,\epsilon_1)} \cdot \R[\Delta]_{(j,\epsilon_2)}\subseteq \R[\Delta]_{(i+j, \epsilon_1\epsilon_2)}$ for all $i,j\in\N$ and $\epsilon_1,\epsilon_2\in\{\pm 1\}$. It is also not hard to see that $\dim_R \R[\Delta]_{(i,+1)}=\dim_R \R[\Delta]_{(i,-1)}=\frac{1}{2} \dim_R \R[\Delta]_i$ for all $i\geq 1$, and that $\R[\Delta]_{(0,+1)}=\R[\Delta]_0=\R$.

One can use the Kind--Kleinschmidt criterion \cite[Theorem III.2.4]{Stanley96} to show that $\R[\Delta]$ has an l.s.o.p.~$\theta_1,\ldots,\theta_d$ such that $\theta_j\in\R[\Delta]_{(1,-1)}$ for all $j$. (For instance, for a cs simplicial polytope $P$ and $\Delta=\partial P$, the l.s.o.p.~described in the proof of Theorem \ref{GLBT} does the job.) Analyzing the exact sequences as in the proof of Theorem \ref{h-numbers-CM}, see eq.~\eqref{exact-seq}, but using our finer grading and, in particular, that $\cdot\theta_j$ maps the degree $(i-1,\pm 1)$ pieces of $\R[\Delta]/(\theta_1,\ldots,\theta_{j-1})$ to the degree $(i,\mp 1)$ pieces, then yields that 
\begin{equation} \label{h_-}
\dim_\R (\R[\Delta]/(\Theta))_{(i,-1)}=\frac{1}{2} \left(h_i(\Delta)-{d \choose i}\right) \quad \mbox{for all } 0\leq i \leq d;
\end{equation}
see \cite{Stanley-87} for more details. This establishes the following result of Stanley \cite{Stanley-87}:

\begin{theorem} \label{cs-CM}
Let $\Delta$ be a cs $\R$-CM complex of dimension $d-1$. Then $h_i(\Delta)\geq {d \choose i}$ for all $i$.
\end{theorem}

We now consider the case of $\Delta=\partial P$, where $P$ is a cs simplicial $d$-polytope (with rational vertices), and $\theta_1,\ldots,\theta_d$ are defined using the coordinates of vertices of $P$ as in the sketch of the proof of Theorem \ref{GLBT}; in particular, $\theta_j\in \R[\Delta]_{(1,-1)}$. Observe that $\omega=\sum_{v\in V}x_v\in\R[\Delta]_{(1,+1)}$ and hence that $\cdot\omega$ maps $(\R[\Delta]/(\Theta))_{(i-1,-1)}$ to $(\R[\Delta]/(\Theta))_{(i,-1)}$. As the map $\cdot\omega: \left(\R[\Delta]/(\Theta)\right)_{i-1} \to \left(\R[\Delta]/(\Theta)\right)_i$ is injective for all $i\leq d/2$, its restriction to $(\R[\Delta]/(\Theta))_{(i-1,-1)}$ is also injective. This, along with
eq.~\eqref{h_-}, completes the proof of Theorem \ref{cs-GLBT-ineq}. It is worth pointing out that Theorem \ref{cs-GLBT-ineq} was extended by Adin \cite{Adin93,Adin95} to the classes of all rational simplicial polytopes with a fixed-point-free linear symmetry of prime and prime-power orders. Furthermore, 	A'Campo-Neuen \cite{ACampo} extended Theorem 7.3 to \emph{toric} $g$-numbers of \emph{all} cs polytopes (including non-rational non-simplicial ones).

The Murai--Nevo proof \cite{MuraiNevo2013} of the equality case of Theorem \ref{GLBT} involves a beautiful blend of tools such as Alexander duality, Stanley--Reisner rings, the Cohen--Macaulay property,  as well as generic initial ideals, and, in particular, Green's crystalization principle \cite[ Proposition 2.28]{Green-gin}. A different proof of both parts of Theorem \ref{GLBT}, including a sharper version of the inequality part, was very recently obtained by Adiprasito \cite[Cor.~6.5 and \S 7]{Adiprasito-toric}. His proof relies on Lee's generalized stress spaces \cite{Lee94}.

When does equality hold in Theorem \ref{cs-GLBT-ineq}? That is, for a fixed $r\leq d/2$, is there a characterization of cs simplicial $d$-polytopes with $g_r={d \choose r}-{d \choose r-1}$? In a recent preprint \cite{KNNZ}, Klee, Nevo, Novik, and Zheng provide such a characterization in the $r=2$ case and a conjectural characterization in the $r>2$ case. Both statements strongly parallel the equality cases of Theorems \ref{LBT} and \ref{GLBT}. If $P$ is a cs simplicial $d$-polytope with $d\geq 4$, then we can apply to $P$ the {\em symmetric stacking operation}: this operation repeatedly attaches (shallow) simplices along antipodal pairs of facets. Note that symmetric stacking preserves both central symmetry and $g_2$.

\begin{theorem} \label{cs-LBT-equality}
Let $P$ be a cs simplicial $d$-polytope with $d \geq 4$.  Then $g_2(P) = {d \choose 2}-d$ if and only if $P$ is obtained from $\C^*_d$ by symmetric stacking.
\end{theorem}


\begin{conjecture} \label{r-stackedness}
Let $P$ be a  cs simplicial $d$-polytope, and assume that $g_r(P) = {d \choose r}-{d \choose r-1}$ for some $3 \leq r \leq \lfloor d/2 \rfloor$. Then there exists a unique polytopal complex $\C$ in $\R^d$ with the following properties: (i) one of the faces of $\C$ is the cross-polytope $\C^*_d$, all other faces of $\C$ are simplices that come in antipodal pairs, (ii) $\C$ is a ``cellulation" of $P$, that is, $\bigcup_{C\in\C}C=P$, and (iii) each element $C\in \C$ of dimension $\leq d-r$ is a face of $P$. 
\end{conjecture}

We mention that \cite[Conjecture 8.6]{KNNZ} provides a more detailed version of the above conjecture and we refer our readers to Ziegler's book \cite[Section 8.1]{Ziegler} for the definition of a {\em polytopal} complex. Conjecture \ref{r-stackedness}, if true, would imply the following weaker statement that is also wide open at present. (The $r=2$ case does hold; this is immediate from Theorem \ref{cs-LBT-equality}.)

\begin{conjecture} \label{g_r}
Let $P$ be a  cs simplicial $d$-polytope. If $g_r(P) = g_r(\C^*_d)$ for some $3 \leq r \leq \lfloor d/2 \rfloor-1$, then $g_{r+1}(P) = g_{r+1}(\C^*_d)$. 
\end{conjecture}

One consequence of Theorems \ref{cs-GLBT-ineq} and \ref{cs-LBT-equality}, along with the fact that the $f$-numbers of spheres are nonnegative linear combinations of the $g$-numbers, is that among all cs simplicial $d$-polytopes with $n$ vertices, a polytope obtained from $\C^*_d$ by symmetric stacking simultaneously minimizes all the face numbers; furthermore, if $d\geq 4$, then such polytopes are the only minimizers. 

While we will not discuss here the proof of Theorem \ref{cs-LBT-equality}, it is worth pointing out that it relies on the rigidity theory of frameworks (some ingredients of this theory are briefly outlined in the next section), and, in particular, on  (1) the theorem of Whiteley \cite{Whiteley-84} asserting that the graph of a simplicial $d$-polytope ($d\geq 3$) with its natural embedding in $\R^d$ is infinitesimally rigid, and on (2) the fact that if $P$ is a cs simplicial $d$-polytope with $g_2(P)= {d \choose 2}-d$, then {\em all} stresses on $P$ must be symmetric (i.e., assign the same weight to each edge and its antipode.) This latter observation follows from work of Stanley \cite{Stanley-87} and Lee \cite{Lee94}, and also from more recent work of Sanyal, Werner, and Ziegler \cite[Theorem 2.1]{Sanyal-et-al}.

To summarize: in contrast with the upper-bound type results, an analog of the GLBT for cs simplicial polytopes, at least its inequality part developed in Theorem \ref{cs-GLBT-ineq}) is a very well-understood part of the story. Furthermore, in the case of the LBT we even have a characterization of the minimizers (see Theorem \ref{cs-LBT-equality}). The part that is still missing is a characterization of cs simplicial $d$-polytopes with $g_r={d \choose r}-{d \choose r-1}$ for $3\leq r \leq \lfloor d/2 \rfloor$. Conjecture \ref{r-stackedness} proposes such a characterization.

\section{The lower bound conjecture for cs spheres and manifolds}
The final part of our story concerns a conjectural analog of the LBT for cs spheres (manifolds and even normal pseudomanifolds) --- a necessary step in our quest for a cs analog of the $g$-conjecture. To this end, it is worth recalling that in the world of simplicial complexes without a symmetry assumption, where at present the GLBT (Theorem \ref{GLBT}) is only known to hold for the class of simplicial polytopes, works of Walkup \cite{Walkup}, Barnette \cite{Barnette-LBT-pseudomanifolds}, Kalai \cite{Kalai87}, Fogelsanger \cite{Fogelsanger88}, and Tay \cite{Tay} show that the LBT (Theorem \ref{LBT}) holds much more generally:

\begin{theorem} \label{LBT-manifolds}
Let $\Delta$ be a simplicial complex of dimension $d-1\geq 3$. Assume further that $\Delta$ is a connected simplicial manifold (or even a normal pseudomanifold). Then $g_2(\Delta)\geq 0$, with equality if and only if $\Delta$ is the boundary complex of a stacked polytope.
\end{theorem}

In view of this result, we strongly suspect that Stanley's inequality on the $g_2$-number of cs polytopes (see Theorem \ref{cs-GLBT-ineq}) as well as the characterization of the minimizers given in Theorem \ref{cs-LBT-equality} continue to hold in the generality of cs simplicial spheres or perhaps even cs normal pseudomanifolds. This conjecture is however wide open at present.

\begin{conjecture} \label{g2_spheres}
Let $\Delta$ be a cs simplicial complex of dimension $d-1\geq 3$. Assume further that $\Delta$ is a simplicial sphere (or a connected simplicial manifold or even a normal pseudomanifold). Then $g_2(\Delta)\geq {d \choose 2}-d$. Furthermore, equality holds if and only if $\Delta$ is the boundary complex of a cs $d$-polytope obtained from the cross-polytope $\C^*_d$ by symmetric stacking.
\end{conjecture}

In the rest of this short section we discuss one potential approach to attacking this conjecture: via the rigidity theory of frameworks. The (now wide) use of this theory in the study of face numbers of simplicial complexes was pioneered by Kalai in his celebrated proof of Theorem \ref{LBT-manifolds} for simplicial manifolds \cite{Kalai87}. Below we review several results and definitions pertaining to this fascinating theory. We refer our readers to Asimow and Roth \cite{AsimowRothI}, \cite{AsimowRothII} for a friendly introduction to this subject. 

Let $G = (V,E)$ be a finite graph.  A map $\p: V \rightarrow \mathbb{R}^d$ is called a \textit{$d$-embedding} of $G$ if $\aff\{\p(v) \ : \ v\in V\}=\R^d$. The graph $G$, together with a $d$-embedding $\p$, is called a $d$-\textit{framework}. An infinitesimal motion of a $d$-framework $(G,\p)$ is a map $\m: V(G)\to \R^d$ such that for any edge $\{u,v\}$ in $G$, 
\[\frac{d}{dt}\Bigr|_{t=0}\big\|(\p(u)+t\m(u))-(\p(v)+t\m(v))\big\|^2=0.\] An infinitesimal motion $\m$ of $(G,\p)$ is trivial if $\frac{d}{dt}\Bigr|_{t=0}\big\|(\p(u)+t\m(u))-(\p(v)+t\m(v))\big\|^2=0$ holds for every two vertices $u,v$ of $G$. (Trivial infinitesimal motions correspond to a start of an isometric motion of $\R^d$.) We say that a $d$-framework $(G, \p)$ is \emph{infinitesimally rigid} if every infinitesimal motion $\m$ of $(G,\p)$ is trivial. 


A \textit{stress} on a $d$-framework $(G,\p)$ is an assignment of weights $\omega = (\omega_e \ : \ e \in E(G))$ to the edges of $G$ such that for each vertex $v$ equilibrium holds: $$\sum_{u\ : \ \{u,v\} \in E(G)} \omega_{\{u,v\}}(\p(v)-\p(u)) = \mathbf{0}.$$ We denote the space of all stresses on $(G,\p)$ by $\Stress(G,\p)$. The stresses on $(G,\p)$ correspond to the elements in the kernel of the transpose of a certain $f_1(G) \times df_0(G)$ matrix $\Rig(G, \p)$ known as the \emph{rigidity matrix} of $(G, \p)$. 

The relevance of rigidity theory to the Lower Bound Theorem is explained by the following fundamental fact that is an easy consequence of the Implicit Function Theorem (see \cite{AsimowRothI} and \cite{AsimowRothII}).

\begin{theorem}  \label{basic-rig-prop} Let $(G, \p)$ be a $d$-framework. Then the following statements are equivalent:
\begin{itemize}
\item  $(G, \p)$ is infinitesimally rigid;
\item   $\rank \Rig(G,\p)=df_0(G)-{d+1 \choose 2}$;
\item $\dim_\R \Stress(G,\p) =f_1(G)-df_0(G)+{d+1 \choose 2}$.
\end{itemize}
\end{theorem}

A graph $G$ is called {\em generically $d$-rigid} if there exists a $d$-embedding $\p$ of $G$ such that $(G, \p)$ is infinitesimally rigid; in this case, the set of infinitesimally rigid $d$-embeddings of $G$ is an open dense subset of the set of all $d$-embeddings. Recall that if $\Delta$ is a $(d-1)$-dimensional simplicial complex, then $g_2(\Delta)=h_2(\Delta)-h_1(\Delta)=f_1(\Delta)-df_0(\Delta)+{d+1 \choose 2}$. The last condition of Theorem \ref{basic-rig-prop} then implies that if the graph (i.e., $1$-skeleton) of $\Delta$ is generically $d$-rigid, then $g_2(\Delta)\geq 0$. 

Two basic but very useful results in rigidity theory are the gluing lemma (\cite[Theorem 2]{AsimowRothII} and \cite[Lemma 11.1.9]{Whiteley96}) and the cone lemma \cite{Whiteley83}. The gluing lemma asserts that if two graphs $G_1$ and $G_2$ are generically $d$-rigid and share at least $d$ vertices, then their union $G_1\cup G_2$ is also generically $d$-rigid, while the cone lemma posits that $G$ is generically $d$-rigid if and only if the graph of the cone over $G$ is generically $(d+1)$-rigid. Additionally, by Gluck's result \cite{Gluck75}, the graph of any simplicial $2$-sphere is generically $3$-rigid. Now, if $\Delta$ is a simplicial $(d-1)$-manifold (with $d\geq 4$) and $F$ is a $(d-4)$-face of $\Delta$, then the link of $F$ in $\Delta$ is a simplicial $2$-sphere, and hence has a generically $3$-rigid graph. This observation, along with the gluing and cone lemmas, allowed Kalai \cite{Kalai87} to prove by induction on $d\geq 4$ that the graph of any simplicial $(d-1)$-manifold is generically $d$-rigid, and thus to establish the inequality part of Theorem \ref{LBT-manifolds} for all simplicial manifolds.

The relationship of infinitesimal rigidity to the Stanley--Reisner ring was worked out in \cite{Lee94}. Specifically, let $\Delta$ be a $(d-1)$-dimensional simplicial complex with vertex set $V$ and let $\p: V\to\R^d$ be a $d$-embedding of the graph of $\Delta$, $G(\Delta)$. Define $d$ linear forms in $\R[x_v \ : \ v\in V]$ by $\theta_j:=\sum_{v\in V} \p(v)_j x_v$ for $j=1,\ldots,d$, where $\p(v)_j$ denotes the $j$-th coordinate of $\p(v)\in\R^d$, and let $\theta_{d+1}:=\sum_{v\in V} x_v$ (cf.~Section 7, and especially the sketch of the proof of Theorem \ref{GLBT}). Theorem 10 of \cite{Lee94} implies that if $(G(\Delta),\p)$ is infinitesimally rigid, then $\dim_\R (\R[\Delta]/(\theta_1,\ldots,\theta_d,\theta_{d+1}))_2=\dim_\R \Stress(G(\Delta),\p)=g_2(\Delta)$. Equivalently,  if $(G(\Delta),\p)$ is infinitesimally rigid, then $\theta_{i+1}: (\R[\Delta]/(\theta_1,\ldots,\theta_i))_1 \to (\R[\Delta]/(\theta_1,\ldots,\theta_i))_2$ is an injection for all $0\leq i\leq d$. 

Assume now that $\Delta$ is centrally symmetric. If there is a $d$-embedding $\p: V(\Delta)\to \R^d$ that respects symmetry and such that $(G(\Delta),\p)$ is infinitesimally rigid, then the previous paragraph along with the $\N\times\Z/2\Z$-grading of $\R[\Delta]$ and relevant computations of Section 7 (e.g., eq.~\eqref{h_-}) shows that $g_2(\Delta)\geq {d \choose 2} -d$. (This also follows from the argument in \cite[Section 2]{Sanyal-et-al}). In particular, the following conjecture, if true, would imply the inequality part of Conjecture \ref{g2_spheres}. 

\begin{conjecture} \label{inf_rig_cs_spheres}
Let $\Delta$ be a $(d-1)$-dimensional cs simplicial complex with an involution $\phi$. Assume further that $\Delta$ is a simplicial sphere (or a connected simplicial manifold or even a normal pseudomanifold) and that $d-1\geq 3$. Then there exists a $d$-embedding $\p: V(\Delta)\to \R^d$ that respects symmetry, i.e., $\p(\phi(v))=-\p(v)$ for each vertex $v$ of $\Delta$, and such that
$(G(\Delta),\p)$ is infinitesimally rigid.
\end{conjecture}

\noindent For instance, if $\Delta$ is the boundary complex of a cs simplicial $d$-polytope $P\subset \R^d$, then by Whiteley's theorem \cite{Whiteley-84}, the natural $d$-embedding of the graph of $P$ qualifies for $\phi$ as in Conjecture~\ref{inf_rig_cs_spheres}\normalsize. 

One of the reasons Conjecture  \ref{inf_rig_cs_spheres} appears to be hard in the general case is that the links of cs complexes are usually not centrally symmetric, and so starting with {\em cs} 2-spheres as the base case (and proceeding by induction via the cone and gluing lemmas) does not work. Conjecture 8.4 of \cite{KNNZ} proposes a statement about rigidity of graphs of (non cs) simplicial $2$-spheres that, if true, will provide an appropriate base case, and imply Conjecture \ref{inf_rig_cs_spheres}. In any case, as a step in our quest for a cs analog of the $g$-conjecture, it would be extremely interesting to shed any light on Conjectures \ref{g2_spheres} and \ref{inf_rig_cs_spheres} as well as to attempt to strengthen the inequality of Conjecture \ref{g2_spheres} in the spirit of results from \cite[Theorem 5.3(i)]{Murai-15} and \cite{MuraiNovik-fundam}. Such strengthened inequalities would provide lower bounds on $g_2$  of a cs simplicial manifold (or even a normal pseudomanifold) $\Delta$ in terms of the first homology or perhaps even in terms of the fundamental group of $\Delta$ and/or of the $\Z/2\Z$-quotient of $\Delta$.

\section{Concluding remarks}
The recent decades made the theory of face numbers into a very active and a rather large field. Consequently, there are quite a few topics we glossed over or omitted in this paper. Among them are the face numbers of general (not necessarily simplicial) cs polytopes. However, one of the conjectures about general cs polytopes we cannot avoid mentioning is Kalai's ``$3^d$-conjecture'' \cite[Conjecture A]{Kalai-89}: it posits that the total number of faces in a cs $d$-polytope (including the empty face but excluding the polytope itself) is at least $3^d$.  If the conjecture is true, there are multiple minimizers: the class of polytopes with exactly $3^d$ faces includes at least all the Hanner polytopes. (If Mahler's conjecture holds, then Hanner polytopes are the minimizers of Mahler volume among all the cs convex bodies.) At present the $3^d$-conjecture is known to hold for all cs simplicial polytopes (this is an immediate corollary of Theorem \ref{cs-CM}) and, by duality, all simple polytopes, as well as for all at most $4$-dimensional cs polytopes \cite{Sanyal-et-al}. The conjecture is wide open in all other cases.

We have also only barely touched on the face numbers of cs simplicial manifolds and instead concentrated on the face numbers of cs simplicial polytopes and spheres. Papers \cite{N05}, \cite[Section 4]{NSw-socle}, and \cite{KleeN}  contain some results pertaining to the Upper Bound Theorem for cs simplicial manifolds as well as to Sparla's conjecture \cite[Conjecture 4.12]{Sparla97}, \cite{Sparla98} on the Euler characteristic of even-dimensional cs simplicial manifolds. 

While we had to skip quite a few of the topics and could not possibly do justice to all of the existing methods, we hope we have conveyed at least some of the essence and beauty of this fascinating subject! We are very much looking forward to progress on the many existing as well as yet unstated problems about cs polytopes and simplicial complexes, and to new interactions between combinatorics, discrete geometry, commutative algebra, and geometric analysis that will lead to this progress!

\subsection*{Acknowledgments}
I am grateful to Steve Klee, Connor Sawaske, Hailun Zheng, and G\"unter Ziegler for numerous comments on the preliminary version of this paper.

{\small
\bibliography{refs}
\bibliographystyle{plain}
}
\end{document}